\documentclass{amsart}
\usepackage{cite}

\usepackage[english]{babel}
\usepackage{amssymb}
\usepackage{latexsym}
\usepackage{xypic}

\newtheorem{theorem}{Theorem}[section]
\newtheorem{lemma}[theorem]{Lemma}
\newtheorem{proposition}[theorem]{Proposition}

\newtheorem{corollary}[theorem]{Corollary}

\theoremstyle{definition}
\newtheorem{definition}[theorem]{Definition}

\theoremstyle{remark}

\def\PSL{{\rm PSL}}
\def\SL{{\rm SL}}

\def\PO{{\rm PO}}
\def\PSO{{\rm PSO}}
\def\O{{\rm O}}
\def\PU{{\rm PU}}
\def\U{{\rm U}}
\def\GL{{\rm GL}}

\def\Kul{{\rm Kul}}
\def\CG{{\rm CG}}

\numberwithin{equation}{section}
\begin{document}

\title{Two dimensional Veronese groups with an invariant ball} 

\author{Angel Cano}
\address{ UCIM UNAM, Unidad Cuernavaca, Av. Universidad s/n. Col. Lomas de Chamilpa,
	C.P. 62210, Cuernavaca, Morelos, M\'exico.}
\email{angelcano@im.unam.mx}

\author{Luis Loeza}
\address{IIT UACJ, Av. del Charro 610 Norte,  Partido Romero, C.P. 32310, Ciudad Ju\'arez, Chihuahua, M\'exico}
\email{  luis.loeza@uacj.mx}

  \thanks{Partially supported by grants of the PAPPIT's project IA100112} 

\maketitle
\begin{abstract}
In this article we characterize the complex hyperbolic groups  that leave invariant a copy of the Veronese curve in $\Bbb{P}^2_{\Bbb{C}}$.
As a corollary  we   get that every discrete compact surface group in $\PO^+(2,1)$  admits  a deformation in 
$\PSL(3,\Bbb{C})$
with a non-empty region of discontinuity which is    not conjugate to a   complex  hyperbolic subgroup.  This provides  a way to construct new examples of Kleinian groups acting on $\Bbb{P}^2_\Bbb{C}$, see  
 \cite{CNS,CS1,SV3,SV1,SV2}.
\end{abstract}

\section*{Introduction}

Back in  the 1990s, Seade and Verjovsky began the  study of discrete groups acting on projective spaces, see \cite{SV3,SV1,SV2}.    Over the years,  new results  have been discovered, see \cite{CNS}.  However, it has been hard to construct groups acting on  $\Bbb{P}^2_\Bbb{C}$ which are neither  virtually affine  nor complex hyperbolic. In this article we use the irreducible representation $\iota$ of $\PSL(2,\Bbb{C})$ into   $\PSL(3,\Bbb{C})$ to produce such groups,  more precisely, we show: 

\begin{theorem}
Let $\Gamma\subset \PSL(2,\Bbb{C})$ be a discrete group of the first kind with non-empty discontinuity region in the Riemann sphere. Then the following claims are equivalent:
\begin{enumerate}
\item The group $\Gamma$ is Fuchsian.
\item  The group $\iota\Gamma$ is complex hyperbolic. 
\item The group $\iota\Gamma$ is  $\Bbb{R}$-Fuchsian.
\end{enumerate}

\end{theorem}

Before we present our next result we should recall the following definition, see \cite{lab} page 30. A group $G$ is called a  {\it compact surface group}, if it is isomorphic to  the fundamental group of a compact orientable topological surface
$\Sigma_g$ of genus $g \geq 2$. 

\begin{theorem} \label{t:main2}
Let  $\Sigma_g$ a compact orientable topological surface
$\Sigma_g$ of genus $g \geq 2$  and  $\rho_0: \Pi_{1}(\Sigma_g)\rightarrow \PO^+(2,1)$ be a faithful discrete representation, where $\PO^+(2,1)$ denotes the projectivization of  identity component   of $\O(2,1)$.  
Then we can find a sequence of discrete faithful   representations $\rho_n:  \Pi_{1}(\Sigma_g)\rightarrow \PSL(3,\Bbb{C})$  such that:
\begin{enumerate}
	\item For  each $n\in \Bbb{N}$ the group $\rho_n(\Pi_1(\Sigma_g))=\Gamma_n$ is a complex  Kleinian group  whose action on $\Bbb{P}^2_{\Bbb{C}}$ is irreducible.
\item  For  each $n\in \Bbb{N}$ the group $\Gamma_n$  is not    conjugate to a subgroup  of $\PU(2,1)$ or
$\PSL(3,\Bbb{R})$.

\item The  sequence  of representations  $(\rho_n)$   converge algebraically   to $\Gamma_0$, {\it  i. e.} $lim_n \rho_n(\gamma)=h$ exists as a projective transformation for all $\gamma\in \Pi_1(\Sigma_g)$   and  
$\Gamma_0=\{h:lim_n\rho_n(\gamma)=h, \gamma\in \Pi_1(\Sigma_g) \}$ compare with the corresponding definition   in \cite{JM}.

\item The  sequence $(\Gamma_n)$ of compact surface groups  converge geometrically  to $\Gamma_0$, {\it  i. e.}  if for every subsequence $(j_n)$ of $(n)$ we get  
	$$\Gamma_0=\{g\in \PSL(3,\Bbb{C}):g=lim_{j_n}\rho_{j_n}(\gamma_n), \gamma_n\in \Pi_1(\Sigma_g)\},$$ compare with the corresponding definition   in \cite{JM}.

	\end{enumerate}

\end{theorem}

\begin{corollary}
There are complex Kleinian groups acting on $\Bbb{P}^2_{\Bbb{C}}$ which are neither conjugate to complex hyperbolic groups nor  virtually affine groups.
\end{corollary}

This paper is organized as follows: in Section \ref{s:recall} we review some general facts and
introduce the notation used throughout the text. In Section \ref{s:gever} we describe  some properties of the Veronese curve which are useful for our purposes.  In Section \ref{s:chvh} we characterize the complex hyperbolic subgroups that leave invariant  a Veronese curve.   In Section \ref{s:riv} we depict  those real hyperbolic subgroups leaving invariant  a Veronese curve. Finally, in Section \ref{s:rep} we show that every discrete compact surface group in $\PO(2,1)^+$ admits  a deformation in $\PSL(3,\Bbb{C})$ which is not conjugate to a   complex  hyperbolic  subgroup and 
has non-empty Kulkarni region of discontinuity.

\section{Preliminaries}  \label{s:recall}

\subsection{Projective geometry}
The complex projective space $\mathbb{P}^2_{\mathbb {C}}$
is defined as
$$\mathbb{P}^{2}_{\mathbb {C}}=(\mathbb {C}^{3}\setminus \{0\})/\Bbb{C}^*,$$
where $\Bbb{C}^*$ acts by  the usual scalar multiplication.
  This is   a  compact connected  complex $2$-dimensional
manifold.
If $[\mbox{}]:\mathbb{C}^{3}\setminus\{0\}\rightarrow\mathbb{P}^{2}_{\mathbb{C}}$ is the quotient map, then a
non-empty set  $H\subset\mathbb{P}^2_{\mathbb{C}}$ is said to
be a line   if there is a  $\mathbb{C}$-linear
subspace  $\widetilde{H}$ in $\mathbb{C}^{3}$ of dimension $2$ such that $[\widetilde{H}\setminus \{0\}]=H$. If $p,q$ are distinct points then  $\overleftrightarrow{p,q}$ is the 
unique complex line passing through them. In this article,  $e_1,e_2,e_{3}$ will denote the standard basis for $\Bbb{C}^{3}$.

\subsection{ Projective  transformations }
 
 The group of projective automorphisms of $\mathbb{P}^{2}_{\mathbb{C}}$ is defined as
$$\PSL(3, \mathbb {C}) \,:=\, \GL({3}, \Bbb{C})/\Bbb{C}^*,$$
where $\Bbb{C}^*$ acts by the usual scalar multiplication. Then 
$\PSL(3, \mathbb{C})$ is a Lie group acting by biholomorphisms on $\Bbb{P}^2_{\Bbb{C}}$;  its
elements are called projective transformations.
We denote   by 
$[[\mbox{  }]]: \GL(3,\mathbb{C})\rightarrow \PSL(3,\mathbb{C})$    the quotient map. Given     
$ \gamma\in\PSL(3, \mathbb{C})$,  we  say that  
$\widetilde\gamma\in\GL(3,\mathbb {C})$ is a {\it lift} of $ \gamma$ if 
$[[\widetilde\gamma]]=\gamma$.\\

\subsection{Complex hyperbolic groups}

In the rest of this paper, we will be interested in studying those subgroups of $\PSL(3,\Bbb{C})$ that preserve  the unitary
complex ball.  We start  by considering the following Hermitian matrix: 
\[
H=
\left (
\begin{array}{lll}
  &&1\\
  &1&\\
1&&
\end{array}
\right ).
\]
We will set
\[
\U(2,1)=\{g\in \GL(3,\Bbb{C}):g^*Hg^*=H\}
\]
\[
\O(2,1)=\{g\in \GL(3,\Bbb{R}):g^t Hg=H\}
\]

and $\langle,\rangle:\Bbb{C}^{3}\rightarrow \Bbb{C}$ the Hermitian form induced by $H$. Clearly, $\langle,\rangle$ has signature 
$(2,1)$ and $\U(2,1)$  is the  group that preserves $\langle,\rangle$, see \cite{goldman}. The  projectivization $\PU(2,1)$ 
preserves the unitary complex ball:
\[
\Bbb{H}^2_\Bbb{C}=\{[w]\in \Bbb{P}^2_{\Bbb{C}}\mid \langle w,w\rangle <0\}.
\]
 Given a subgroup $\Gamma\subset\PU(2,1)$, we define the following notion of limit set, as in  \cite{CG}.
 
\begin{definition}

Let $\Gamma\subset \PU(2,1)$, then its Chen--Greenberg limit set   is  $\Lambda_{\CG}(\Gamma):= \bigcup \overline{\Gamma x}\cap \partial \Bbb{H}^2_\Bbb{C}$
where the union on the right runs over all points $x\in \Bbb{H}^2_\Bbb{C}$.
\end{definition}

 As in  the Fuchsian groups case,  $\Lambda_{\CG}(\Gamma)$ has either 1, 2 or infinitely many points.  A group is said to be non-elementary if  $\Lambda_{\CG}(\Gamma)$ has 
infinitely many points, and in that case it does not depend on the choice of orbit, {\it i.e.}  $\Lambda_{\CG}(\Gamma):= \overline{\Gamma x}\cap \partial \Bbb{H}^2_\Bbb{C}$
where $x\in \Bbb{H}^2_\Bbb{C}$ is any point.
\\
 
 \subsection{ Pseudo-projective transformations}

The space of linear transformations from $\Bbb{C}^{3}$ to $\Bbb{C}^{3}$, denoted by  $M(3,\Bbb{C})$,   is a complex linear space  of dimension
 $9$, where $\GL(3,\Bbb{C})$ is an open dense set in $ M(3,\Bbb{C})$. Then
$\PSL(3,\Bbb{C})$ is an open dense set in  
$QP(3,\Bbb{C})=(M(3,\Bbb{C})\setminus\{0 \})/\Bbb{C}^*$ called in \cite{CS} the space of pseudo-projective   maps. Let 
$\widetilde{M}:\mathbb{C}^{3}\rightarrow\mathbb{C}^{3}$ be a non-zero
 linear transformation. Let   $Ker(\widetilde M)$ be its kernel and 
  $Ker([[\widetilde M]])$
 denote its projectivization.
 Then  $\widetilde{M}$ induces a well defined map 
  $[[\widetilde M]]:\mathbb {P}^{2}_\mathbb {C}\setminus Ker([[\widetilde M]]) \rightarrow\mathbb {P}^{2}_\mathbb {C}$   by
 $$[[\widetilde M]]([v])=[\widetilde M(v)].$$
 The following  result provides a relation between convergence in $QP(3,\Bbb{C})$ viewed as points in a projective space and the 
 convergence viewed as functions.

\begin{proposition} [See \cite{CS}]  \label{p:completes}
Let  $(\gamma_m)_{m\in \mathbb {N}}\subset \PSL(3,\mathbb {C})$
be a sequence of distinct elements. Then: 
\begin{enumerate}
\item  There is a subsequence 
$(\tau_m)_{m\in \mathbb {N}}\subset(\gamma_m)_{m\in\mathbb{N}}$ and  
a $\tau_0\in M(3,\Bbb{C})\setminus\{0\}$ such that
 $\tau_m\xymatrix{\ar[r]_{m\rightarrow\infty}&}\tau_0$ as points in $QP(3,\Bbb{C})$.

\item If $(\tau_m)_{m\in \mathbb {N}}$ is the sequence  given by the previous part of this lemma, then 
$\tau_m\xymatrix{\ar[r]_{m\rightarrow\infty}&}\tau_0$, as   functions,   uniformly on compact sets of 
 $\mathbb{P}^n_\mathbb{C}\setminus Ker(\tau_0)$. Moreover, the equicontinuity set of  $\{\tau_m\vert m\in \Bbb{N} \}$ is $\Bbb{P}^n\setminus Ker(\tau_0)$.

\end{enumerate}
\end{proposition}
\subsection{Kulkarni's limit set}
When we look at the action of a group acting on a general topological space, in  general there is no natural notion of limit set. 
A nice starting point is Kulkarni's limit set. 

\begin{definition}[see   \cite{kulkarni} ] \label{d:lim}
 Let $\Gamma\subset\PSL(n+1,\mathbb{C})$ be a subgroup. We  define
 
\begin{enumerate}
\item the set 
$\Lambda(\Gamma)$ to be the closure of the set  of cluster points  of
$\Gamma z$  as  $z$ runs  over  $\mathbb{P}^n_{\mathbb{C}}$,
\item  the set $L_2(\Gamma)$ to be  the closure of cluster  points of 
$\Gamma K$  as $K$ runs  over all  the compact sets in
$\mathbb{P}^n_{\mathbb{C}}\setminus \Lambda(\Gamma)$,
\item and \textit{ Kulkarni's limit set } of $\Gamma$  to be  
$$\Lambda_{\Kul}(\Gamma)=\Lambda(\Gamma)\cup L_2(\Gamma),$$ 
\item \textit{ Kulkarni's discontinuity region} of $\Gamma$  to be
$$\Omega_{\Kul}(\Gamma)=\mathbb{P}^n_{\mathbb{C}}\setminus\Lambda_{\Kul}(\Gamma).$$
\end{enumerate}

\end{definition}

Kulkarni's limit set has the following properties.  For a more  detailed  discussion of this in the two-dimensional setting,
see \cite{CNS}. 

\begin{proposition}[See \cite{kulkarni,CS,BCN,CNS}] \label{p:pkg}
Let    $\Gamma\subset \PSL(3,\Bbb{C})$  be a complex  Kleinian group. Then:

\begin{enumerate}

\item The sets\label{i:pk2}
$\Lambda_{\Kul}(\Gamma),\,\Lambda(\Gamma),\,L_2(\Gamma)$
are  $\Gamma$-invariant closed sets. 

\item \label{i:pk3} The group $\Gamma$ acts properly 
  discontinuously on  $\Omega_{\Kul}(\Gamma)$. 

\item \label{i:pk4} If $\Gamma$ does not have any projective invariant subspaces, then  $$\Omega_{\Kul}(\Gamma)=Eq(\Gamma).$$ Moreover, $\Omega_{\Kul}(\Gamma)$ is complete Kobayashi hyperbolic and  is the largest open set  on which the group acts properly discontinuously.
\end{enumerate}
\end{proposition}

\section{ The Geometry of the  Veronese Curve} \label{s:gever}

Now let us define
the   Veronese embedding. Set 
\[
\begin{array}{l}
\psi:\Bbb{P}^1_\Bbb{C}\rightarrow \Bbb{P}^2_\Bbb{C}\\
\psi([z,w])=[z^2,2zw, w^2].
\end{array}
\]

Let us  consider $\iota: \PSL(2,\Bbb{C})\rightarrow \PSL(3,\Bbb{C})$ given by
\[
\iota\left(\frac{az+b}{cz+d}\right )=\left [\left [
\begin{array}{lll}
a^2&ab&b^2\\
2ac&ad+bc&2bd\\
c^2&dc&d^2\\
\end{array}
\right ]\right ].
\]
 Trivially,  $\iota$ is well defined. Note that this map is  induced by the canonical action of $\SL(2,\Bbb{C})$ on the space of homogeneous polynomials of degree two in two complex variables.

\begin{lemma} \label{l:mor}
The map  $\iota$ is an injective group morphism.
\end{lemma}
\begin{proof}
Let 
 \[
A=\left [\left [
\begin{array}{ll}
a&b\\
c&d\\
\end{array}
\right ]\right ],\,B=\left [\left [
\begin{array}{ll}
 e&f\\
g&h\\
\end{array}
\right ]\right ]\in\PSL(2,\Bbb{C}).
\] 
Then 
\begin{tiny}
\[
\begin{array}{ll}
\iota (AB)
&=\iota\left [\left [
\begin{array}{ll}
 ae+bg&af+bh\\
ce+dg&cf+dh\\
\end{array}
\right ]\right ]\\
&=\left [\left [
\begin{array}{lll}
 (ae+bg)^2&(ae+bg)(af+bh)&( af+bh)^2\\
2(ae+bg)(ce+dg)&(cf+dh)(ae+bg)+(af+bh)(ce+dg)&2(af+bh)(cf+dh)\\
(ce+dg)^2&2(ce+dg)(cf+dh)&(cf+dh)^2\\
\end{array}
\right ]\right ]\\
&=\left [\left [
\begin{array}{lll}
 a^2&ab&b^2\\
2bc&ad+bc&2bd\\
c^2&cd&d^2\\
\end{array}
\right ]\right ]\left [\left [
\begin{array}{lll}
 e^2&ef&f^2\\
2eg&eh+fg&2fh\\
g^2&gh&h^2\\
\end{array}
\right ]\right ]\\
&=\iota(A)\iota(B).
\end{array}
\]
\end{tiny}
Therefore  $\iota$ is a group  morphism. Now suppose 
 $A=[[a_{ij}]]\in \PSL(2,\Bbb{C})$ is such  that $\iota(A)=Id$. Then
\[
\left [\left [
\begin{array}{lll}
 a^2_{11}&a_{11}a_{12}&a^2_{12}\\
2a_{12}a_{21}&a_{11}a_{22}+a_{12}a_{21} &2a_{12}a_{21}\\
a_{21}^2&a_{21}a_{22}&a_{22}^2\\
\end{array}
\right ]\right ]
=\left [\left [
\begin{array}{lll}
1&0&0\\
0&1&0\\
0&0&1\\
\end{array}
\right ]\right ]
\]
and so we conclude  $a_{12}=a_{21}=0$.  Since   $a_{11}a_{22}-a_{12}a_{21}=1$, we deduce $a_{11}^2=a_{22}^2=1$, {\it i.e.} $A=Id$, which concludes the proof.
\end{proof}

\begin{proposition}
The morphism  $\iota$  is  type preserving. In particular, if 
$\Gamma\subset\PSL(2,\Bbb{C})$ is a discrete subgroup, we must have  $\iota(\Gamma)$ is a discrete group such that each element is strongly loxodromic.
\end{proposition}

Here, by type preserving, we mean that  $\iota$ carries elliptic elements into elliptic elements, and similarly for 
loxodromic and parabolic elements.

\begin{proof}
Consider 
\[
A=
\left [\left [
\begin{array}{ll}
a&0\\
0&a^{-1}
\end{array}
\right ]\right ],\,B=
\left [ \left [
\begin{array}{ll}
1&1\\
0&1
\end{array}
\right ]\right ]\in\PSL(2,\Bbb{C}).
\] 
A straightforward calculation shows  
\[
\iota(A)=
\left [\left [
\begin{array}{lll}
a^2&0&0\\
0&1&0\\
0&0&a^{-2}
\end{array}
\right ]\right ],\,\iota(B)=
\left [ \left [
\begin{array}{lll}
1&1&1\\
0&1&0\\
0&0&1\\
\end{array}
\right ]\right ].
\]
This shows that $\iota$ is type preserving.
Now let
\[
A_n=
\left [\left [
\begin{array}{ll}
a_n&b_n\\
c_n&d_n\\
\end{array}
\right ]\right ]\in\PSL(2,\Bbb{C})
\]
be a sequence such that  $\iota(A_n)\xymatrix{\ar[r]_{n\rightarrow\infty}&}Id$. Then

\[
 \left [\left [
\begin{array}{lll}
a^2_n&a_nb_n&b^2_n\\
2a_nc_n& a_nd_n+b_nc_n&2b_nd_n\\
c_n^2&d_nc_n&d^2_n\\
\end{array}
\right ] \right ]
\xymatrix{\ar[r]_{n\rightarrow\infty}&}Id.
\]
Therefore the $(a^2_n),(d^2_n)$ are bounded and bounded away from $0$,   $b_n^2\xymatrix{\ar[r]_{n\rightarrow\infty}&}0$, and    
$c_n^2\xymatrix{\ar[r]_{n\rightarrow\infty}&}0$, which is a contradiction.    
\end{proof}

\begin{lemma}
Let $g\in\PSL(3,\Bbb{C}) $ be such that  $g$ fixes four points in general  position. Then $g=Id$.
\end{lemma}
\begin{proof}
We can assume that   the four  points in general position  fixed by $g$ are $\{e_1,e_2,e_3,p\}$. Then 
\[
g=
\left [\left [
\begin{array}{lll}
a_1&0&0\\
0&a_2&0\\
0&0& a_3
\end{array}
\right ]\right ].
\]

Since  $p,e_1,e_2,e_3$ are in general position, we conclude    $p=[b_1,b_2,b_3]$ where   $b_1b_2b_3\neq 0$. On the other hand, since  $p$ is fixed we deduce  
$$[b_1,b_2,b_3]=[a_1b_1,a_2b_2,a_3b_3],$$
therefore there is an $r\in\Bbb{C}^*$ such that  $b_i=ra_ib_i$. In consequence $a_1=a_2=a_3$, which concludes the proof.   
\end{proof}

\begin{lemma}
The Veronese curve has  four points in general position.
\end{lemma}
\begin{proof}
A straightforward calculation shows  that $[1,0,0], [0,0,1],[1,2,1],[1,2i,-1]$ are points on the 
Veronese curve. In order to conclude the proof, it is enough to observe 
\[
\left \vert
\begin{array}{lll}
1&1&1 \\
0&2&2i\\
0&1& -1
\end{array}\right\vert=-2-2i,\,\hbox{and}\,\left\vert
\begin{array}{lll}
0&1&1\\
0&2&2i\\
1&1&-1
\end{array}\right\vert=-2+2i.
\]
\end{proof}

\begin{proposition}
The subgroup of  $\PSL(3,\Bbb{C})$ leaving  $\psi(\Bbb{P}_\Bbb{C}^1)$ invariant  is 
$\iota(\PSL(2,\Bbb{C}))$.
\end{proposition}
\begin{proof}
First, let us prove that   $Ver=\psi(\Bbb{P}_\Bbb{C}^1)$ is invariant under  $\iota(\PSL(2,\Bbb{C}))$. Let $A=[[a_{ij}]]\in \PSL(2,\Bbb{C})$. Then 
\[
\iota
\left [\left [ 
\begin{array}{ll}
a_{11}&a_{12}\\
a_{21}&a_{22}\\
\end{array}
\right ]\right ]
\left [
\begin{array}{l}
x\\
2xy\\
y^2\\
\end{array}
\right ]=
\left [
\begin{array}{l}
(a_{11}x+a_{12}y)^2 \\
2(a_{21}x+a_{22}y)(a_{11}x+a_{12}y)\\
(a_{21}x+a_{22}y)^2\\
\end{array}
\right ],
\]
and so  $Ver$ is invariant under  $\iota \PSL(2,\Bbb{C})$ and 
 the following diagram commutes.
\begin{equation} \label{e:aut}
\xymatrix{
\Bbb{P}_\Bbb{C}^1 \ar[r]^\gamma \ar[d]^\psi         & \Bbb{P}_\Bbb{C}^1 \ar[d]^\psi \\
Ver \ar[r]^{\iota \gamma}  & Ver
} 
\end{equation}
Now let $\tau\in\PSL(3,\Bbb{C})$ be an element which leaves $Ver$ invariant. Define
\[
\begin{array}{l}
\widetilde{\tau}:\Bbb{P}_\Bbb{C}^1\rightarrow \Bbb{P}_\Bbb{C}^1 \quad.\\
 \widetilde{\tau}(z)=\psi^{-1}(\tau(\psi(z))).
\end{array}
\]
Clearly $\widetilde{\tau}$ is well defined and  biholomorphic, thus 
$\widetilde{\tau}\in \PSL(2,\Bbb{C})$ and the following diagram commutes.
\begin{equation*}
\xymatrix{
\Bbb{P}_\Bbb{C}^1 \ar[r]^{\widetilde \tau}\ar[d]^\psi         & \Bbb{P}_\Bbb{C}^1 \ar[d]^\psi \\
Ver \ar[r]^{\tau}  & Ver
} 
\end{equation*}
From diagram  \ref{e:aut}, we conclude that  
$\tau\mid_{Ver}=\iota\widetilde\tau\mid_{Ver}$. Since  the  Veronese curve has  four points in general   position, we conclude  $\tau=\iota\widetilde \tau$ in $\Bbb{P}_\Bbb{C}^2$, which concludes the proof.
\end{proof}

\begin{lemma}\label{l:ltanver}
Given $[1,k]\in \Bbb{P}^1_\Bbb{C}$,  the tangent  line to $Ver$ at $\psi[1,k]$, denoted $T_{\psi[1,k]}Ver$, is given by
$$T_{\psi[x,y]}Ver=\{[x,y,z]\in \Bbb{P}^2_\Bbb{C}\vert z=ky-k^2x\}.$$
\end{lemma}

\begin{proof} Let us consider the  chart 
	$(W_1=\{[x,y,z]\in\Bbb{P}^2_\Bbb{C}\vert x\neq 0\},\phi_1:W_1\rightarrow\Bbb{C}^2) $ of 
	$\Bbb{P}^2_\Bbb{C}$ where $\phi_1[x,y,z]=(yx^{-1},zx^{-1})$ and 
$(W_2=\{[x,y]\in\Bbb{P}^1_\Bbb{C}\vert x\neq 0\},\phi_2:W_2\rightarrow\Bbb{C}^1)$ of $\Bbb{P}^1_\Bbb{C}$ where $\phi_1[x,y]=yx^{-1}$.  Let us define
\[
\begin{array}{l}
\phi:\Bbb{C}\rightarrow \Bbb{C}^2\\
\phi(z)=\phi_1(\psi(\phi_2^{-1}( z)))
\end{array}.
\]

A  straightforward calculation shows that $\phi(z)=(2z,z^2)$, thus the tangent space to the curve $\phi$ at $\phi(k)$ is $\Bbb{C}(1, k)+(2k,k^2)$. Therefore the tangent line to $Ver$ at $[1,2k,k^2]$ is $\overleftrightarrow{[1,2k,k^2], [1,2k+1,k+k^2]}$. A simple verification shows 

$$T_{\psi[x,y]}Ver=\{[x,y,z]\in \Bbb{P}^1_\Bbb{C} \vert z=ky-k^2x\}.$$
\end{proof}

\begin{lemma}\label{l:3gen}
Let $\Gamma\subset\PSL(2,\Bbb{C})$ be a non-elementary subgroup and   
$x,y,z\in \Lambda(\Gamma)$  be distinct points, then  the lines  $T_{\psi(x)}Ver,T_{\psi(y)}Ver,T_{\psi(z)}Ver$ are in general  position.
\end{lemma}
\begin{proof}
Let us assume that  $[1,0],[0,1]\notin \Lambda(\Gamma)$. Then there are $k,r,s\in \Bbb{C}$ such that 

 \[
\begin{array}{l}
\psi(x)=[1,2k,k^2] \\
 \psi(y)=[1,2r,r^2] \\
\psi(z)=[1,2s,s^2].
\end{array}
\]
From Lemma \ref{l:ltanver} we know  
 \[
\begin{array}{l}
T_{\psi(x)} Ver=\{
[x,y,z]
\in \Bbb{P}^1_\Bbb{C} \vert z=ky-k^2x\} \\
T_{\psi(y)} Ver=\{
[x,y,z]
\in \Bbb{P}^1_\Bbb{C} \vert z=ry-r^2x\} \\
T_{\psi(z)} Ver =\{
[x,y,z]
\in \Bbb{P}^1_\Bbb{C} \vert z=sy-s^2x\}. 
\end{array}
\]
Since
\[
\left \vert
\begin{array}{lll}
k^2&-k&1\\
r^2&-r&1\\
s^2&-s&1\\
\end{array}\right\vert=(s-r)(k-s)(k-r)\neq 0
\]
we conclude the proof.
\end{proof}

\begin{lemma} \label{l:pseudo}
Let $(\gamma_n)\subset \PSL(2,\Bbb{C})$ be a sequence of distinct elements such that 
$\gamma_n\xymatrix{\ar[r]_{\rightarrow\infty}&}x$ uniformly on compact sets of $\Bbb{P}^1_\Bbb{C}\setminus\{y\}$. Then 
$\iota\gamma_n\xymatrix{\ar[r]_{\rightarrow\infty}&}\psi(x)$ uniformly on compact sets of $\Bbb{P}^2_\Bbb{C}\setminus T_{\psi(y)}Ver$.
\end{lemma}
\begin{proof}
Let us assume that $\gamma_n=\left [\left [a_{ij}^{(n)}\right ]\right ]$. Note that we can assume  $a_{ij}^{(n)}\xymatrix{\ar[r]_{n\rightarrow\infty}&}a_{ij}$ and $\sum_{i,j=1}^2\mid a_{ij} \mid\neq 0$. Then 
$\gamma_n\xymatrix{\ar[r]_{n\rightarrow\infty}&}\gamma=\left [\left [a_{ij}\right ]\right ]$ uniformly on compact sets of 
$\Bbb{P}^1_{\Bbb{C}}\setminus Ker(\gamma)$, thus  
$Ker(\gamma)=\{y\}$ and  $Im(\gamma)=\{x\}$. Therefore  there is a $k\in \Bbb{C}^*$ such that $x= [1,k]$, 
thus  $a_{11}=-ka_{12}$ and  $a_{21}=-ka_{22}$. In consequence 
\[
\iota\gamma_n
 \xymatrix{ \ar[r]_{n \rightarrow\infty}&}
B=
 \left [  \left [
\begin{array}{lll}
k^2a_{12}^2&-ka_{12}^2& a_{12}^2\\
2k^2a_{12}a_{22}&-2ka_{12}a_{22}&2a_{12}a_{22}\\
k^2a_{22}^{2}&-ka_{22}^2&a_{22}^2\\
\end{array}
\right ]\right ].
\]

A simple calculation shows 
that $Ker(B)$ is the line 
$\ell=\overleftrightarrow{ [e_1-k^2e_3],[e_2+ke_3]}$. Also it is not hard to check that 
\[\ell=\{[x,y,z]\vert k^2x-ky+z=0 \},\]
which concludes the proof.
\end{proof}

\begin{proposition}
Let  $\Gamma\subset\PSL(2,\Bbb{C})$ be a non-elementary group. Then  $\iota(\Gamma)$  does not have invariant subspaces in $\Bbb{P}^2_\Bbb{C}$.
\end{proposition}
\begin{proof}
Let us assume that there is a complex line  $\ell$ invariant under  $\iota(\Gamma)$. By  Bézout's theorem  $Ver\cap \ell$ has either   one or two  points. From the following commutative diagram 
\begin{equation*}
\xymatrix{
\Bbb{P}_\Bbb{C}^1 \ar[r]^{ \tau}\ar[d]^\psi         & \Bbb{P}_\Bbb{C}^1 \ar[d]^\psi \\
Ver \ar[r]^{\iota\tau}  & Ver
}
\end{equation*}
where  $\tau\in\Gamma$, we deduce that  $\Gamma$ leaves 
$\psi^{-1}(Ver\cap\ell)$ invariant. Therefore  $\Gamma$ is an elementary group, which is a contradiction, thus $\iota\Gamma$ does not have invariant lines in $\Bbb{P}^2_\Bbb{C}$. 
Finally, if there is a point 
 $p\in\Bbb{P}_\Bbb{C}^2$ fixed by  $\iota\Gamma$, then by  Lemmas \ref{l:3gen},  \ref{l:eq} and \ref{l:pseudo}, 
 there is a sequence of distinct elements
  $(\gamma_m)_{m\in\Bbb{N}}\subset\Gamma$  and a pseudo-projective transformation 
 $\gamma\in QP(3,\Bbb{C})$ such that $\iota\gamma_m\xymatrix{\ar[r]_{m\rightarrow\infty}&}\gamma $ and  $Ker(\gamma)$ is a complex line not containing $p$.
Since  $p$ is invariant and outside $Ker(\gamma)$ we conclude $\{p\}=Im(\gamma)$. 
On the other hand, by Lemma \ref{l:pseudo} we deduce  $p\in Ver$. Therefore   
$\Gamma$ is elementary, which is a contradiction.
\end{proof}

The following theorem follows easily from the previous discussion.

\begin{theorem} \label{l:eq}
Let   $\Gamma$ be a discrete subgroup  of   $\PSL(2,\Bbb{C})$. Then 
$$\Bbb{P}_\Bbb{C}^2\setminus Eq(\iota(\Gamma))=\bigcup_{z\in\Lambda(\Gamma)}T_{\psi(z)}(\psi(\Bbb{P}_\Bbb{C}^)). $$
 Moreover $\Omega_{\Kul}(\iota\Gamma)=Eq(\iota(\Gamma))$ is Kobayashi hyperbolic,  pseudo-convex, and 
is the largest open set on which  $\Gamma$ acts properly  discontinuously.
\end{theorem}

\section{Complex Hyperbolic Groups Leaving $Ver$ Invariant} \label{s:chvh}

 In this section we   characterize  the subgroups of $\PU(2,1)$  that leave  invariant a  projective translation   of  the Veronese curve $Ver$. We need some preliminary lemmas.

\begin{lemma}  \label{l:semialg} Let  $B$ be a complex ball. Then   
$$Aut(BV)=\{g\in\PSL(3,\Bbb{C})\vert g\in \iota\PSL(2,\Bbb{C}),gB=B\}$$ is a   semi-algebraic group.
\end{lemma}
\begin{proof}
Since  $\iota(\PSL(2,\Bbb{C}))$ and   
$\PU(2,1)$ are simple Lie groups with  trivial centers, we deduce that  they are  semi-algebraic groups (see \cite{semi}). Thus the sets 
\[ 
\begin{array}{l}
\{(g,h,gh): g,h\in Aut(BV)\}\\
\{(g,g^{-1}): g\in Aut(BV)\}
\end{array}
\]
are semi-algebraic sets. Therefore  $Aut(BV)$ is a 
semi-algebraic group.
\end{proof}

\begin{lemma} \label{c:liedim}
Let $\Gamma\subset\PSL(2,\Bbb{C})$ be a discrete non-elementary group such that  $\iota\Gamma$ leaves invariant a complex ball  $B$. Then:

\begin{enumerate}
\item \label{l:1} The group  $Aut(BV)$ is a  Lie group of  positive dimension.
\item \label{l:2} We have $\psi\Lambda(\Gamma)\subset Ver\cap\partial B$.

\item \label{l:3} Set $C=\partial B\cap Ver$. Then the set $\psi^{-1}(C)$ is an  algebraic curve of  degree at most four.
  
\item \label{l:4} The group    $\iota^{-1}Aut(BV)$ can be  conjugated  to a subgroup of   $Mob(\hat{\Bbb{R}})$, where 
$Mob(\hat{\Bbb{R}})=\{\gamma\in\PSL(2,\Bbb{C}):\gamma(\Bbb{R}\cup\{\infty\})=\Bbb{R}\cup\{\infty\}\}$.

\item \label{l:5} The set  $\psi^{-1}(C)$ is a circle in the  Riemann sphere.

\item \label{l:6} The set $C$ is an $\Bbb{R}$-circle, {\it i.e.} 
$C=\gamma(\partial\Bbb{H}^2_{\Bbb{C}}\cap\Bbb{P}^2_{\Bbb{R}})$, where $\gamma\in\PSL(3,\Bbb{C})$ is some element satisfying   $\gamma(\Bbb{H}^2_{\Bbb{C}})=B$.

\item \label{l:7} The set  $Ver\cap (\Bbb{P}^2_{\Bbb{C}}\setminus\overline{B})$ is non-empty.

\item \label{l:8} The set  $Ver\cap  B$ is non-empty.

\end{enumerate}
\end{lemma}

\begin{proof}
 Let us start by showing  (\ref{l:1}). Since 
 $Aut(BV)$ is semi-algebraic, we deduce that  it is a Lie group with  a finite number of connected  components   (see \cite{semi}). On the other hand, since   $Aut(BV)$ contains a discrete subgroup, we conclude $Aut(BV)$ has positive dimension.\\

Now let us prove  part (\ref{l:2}). Let $x\in \Lambda(\Gamma)$. Then there is a sequence  $(\gamma_n)\subset \Gamma$ of distinct elements  such that  $\gamma_n\xymatrix{\ar[r]_{m\rightarrow\infty}&}x$ uniformly on compact sets of $\widehat{\Bbb{C}}\setminus \{x\}$. From  Lemma \ref{l:pseudo}  we know that  
$\iota\gamma_n\xymatrix{\ar[r]_{m\rightarrow\infty}&}\psi(x)$ uniformly on compact sets of 
$\Bbb{P}_\Bbb{C}^2\setminus T_{\psi(x)}Ver$, thus  $\psi x\in \partial B$ and 
$T_{\psi(x)}Ver$  is tangent  to  $\partial(B)$ at $x$. This concludes the proof.\\

Now let us prove part  (\ref{l:3}).   Since $\Gamma$ preserves the ball $B$,  there is  a Hermitian matrix $A=(a_{ij})$ with signature $(2,1)$ such that 
$B=\{[x]\in\Bbb{P}^2_{\Bbb{C}}:\overline{x}^t Ax<0\}$. 
 Without loss of generality, we may assume that  $[0,0,1]\notin C=\partial(B)\cap Ver$. Thus for each   $x\in C$, there is a unique  $z\in \Bbb{C}$ such that  
  $x=[1,2z,z^2]=\psi [1,z]$ and 
$(1,2\bar z,\bar{z}^2)^tA(1, 2z,z^2)=0$.
A straightforward calculation shows that  $(1,2\bar z,\bar{z}^2)^tA(1, 2z,z^2)=0$ is equivalent  to
\begin{equation}\label{e:cuadrica}
a_{11}+4Re(a_{12}z)+2Re(a_{13}z^2)+a_{33}\vert z\vert^4+4\vert z\vert^2 Re(a_{23}z)+4\vert z \vert^2a_{22}=0.
\end{equation}
Taking $z=x+iy$ and  $a_{ij}=b_{ij}+ic_{ij}$, Equation  (\ref{e:cuadrica})   can be written as
\[
\begin{array}{l}
a_{11}+4(b_{12}x-c_{12}y)+2(b_{13}(x^2-y^2)-2c_{13}xy)+a_{33}(x^2+y^2)^2+\\
+4(x^2+y^2)( b_{23}x-c_{23}y)+4(x^2+y^2)a_{22}=0,
\end{array}
\]
which proves the assertion.\\

Let us prove part  (\ref{l:4}). Since $\iota^{-1}Aut(BV)$ is a Lie group with positive dimension 
containing a  non-elementary  discrete subgroup, we deduce that (see \cite{CS1}) 
$\iota^{-1}Aut(BV)$ can be conjugated either to  $\PSL(2,\Bbb{C})$ or  a subgroup of   $Mob(\hat{\Bbb{R}})$. On the other hand,  we know that    $\PSL(2,\Bbb{C})$ acts transitively on the  Riemann sphere, but 
$\iota^{-1}Aut(BV)$ leaves an  algebraic curve invariant, plus a point, therefore  $\iota^{-1}Aut(BV)$ is conjugate  to a subgroup of  
 $Mob(\hat{\Bbb{R}})$, which concludes the proof.\\

Let us prove part  (\ref{l:5}). We know that $C$ is 
  $Aut(BV)$-invariant and by part (\ref{l:3}) of the present lemma  $\psi^{-1}C$  is an 
algebraic curve. Thus by  Montel's Lemma we conclude that  
$\Lambda_{Gr}\iota^{-1}Aut(BV)\subset \psi^{-1}C$, where  $\Lambda_{Gr}\iota^{-1}Aut(BV)$ is the  Greenberg limit set of $\iota^{-1}Aut(BV)$,  see \cite{CS1}. Finally by part  (\ref{l:3}), we know that   
$ \iota^{-1}Aut(BV)$ is conjugate to a subgroup of  $Mob(\hat{\Bbb{R}})$, therefore    
$ \Lambda_{Gr}\iota^{-1}Aut(BV)$ is a circle in the  Riemann sphere and $\Lambda_{Gr}\iota^{-1}Aut(BV) = \psi^{-1}C$.\\

In order to prove part  (\ref{l:6}), observe that  after a projective  change of coordinates we can assume that $\psi^{-1}C=\hat{\Bbb{R}}$. Thus   
$C=\psi  \hat{\Bbb{R}}=\{[z^2,2zw,w^2]:z,w\in \Bbb{R}, \vert a\vert +\vert b\vert \neq 0\}$. The following claim concludes the proof.\\

Claim.  The sets   $C$ and      $\partial \Bbb{H}^1_{\Bbb{R}}=\{[x,y,z]\in\Bbb{P}^2_{\Bbb{R}}:x^2+y^2=z^2\}$ are  projectively equivalent. 
 Let   $\gamma\in PSL(3,\Bbb{R})$, be the projective transformation  induced by:
 \[
\widetilde  \gamma=
 \begin{pmatrix}
 1 & 0 & -1\\
 0 & 1 &0\\
 1 & 0 &1
 \end{pmatrix}.
 \]
Given  $[p]=[x^2,2xy, y^2]\in C$, we get $\gamma(p)=(x^2 - y^2, 2 xy, x^2 + y^2)$ and     
$$
(x^2 - y^2)^2+ (2 xy)^2=( x^2 + y^2)^2.
$$
Thus $\gamma C\subset \partial \Bbb{H}^1_{\Bbb{R}}$. Since  $C$ is a  compact,  connected and contains more than two points we conclude that $\gamma$ is a projective equivalence between $C$ and $\partial \Bbb{H}^1_{\Bbb{R}}$. \\

Now we prove part (\ref{l:7}). Let $x\in B$. Then  $x^{\bot}$ is a complex line in  $\Bbb{P}_\Bbb{C}^2\setminus \bar{B}$; by  Bézout's theorem we know  
$Ver\cap x^\bot$ is non-empty, thus   
$Ver\cap(\Bbb{P}^2_\Bbb{C}\setminus\bar{B})\neq\emptyset$.\\

Finally, let us prove part (\ref{l:8}). After conjugating by an element in 
$\iota\PSL(3,\Bbb{C})$ we can assume that  $[0,0,1]\notin\partial B$. 
Let $A=(a_{ij})$ be the Hermitian matrix introduced in part (\ref{l:3}) of the present lemma. Clearly $a_{33}\neq 0$. Now let $F:\Bbb{R}^2\rightarrow \Bbb{R}$ be given by 

\begin{scriptsize}
\[
F(x,y)=a_{11}+4(b_{12}x-c_{12}y)+2(b_{13}(x^2-y^2)-2c_{13}xy)+a_{33}(x^2+y^2)^2
+4(x^2+y^2) ( b_{23}x-c_{23}y+a_{22}).
\]
\end{scriptsize}
Thus by part  (\ref{l:5}) of this lemma we know $\psi^{-1}C=F^{-1}0$ is a circle. Moreover
\[
\begin{array}{l}
\iota F^{-1}\Bbb{R}^+=Ver\cap\Bbb{P}_\Bbb{C}^2\setminus \bar{B}.\\
\iota F^{-1}\Bbb{R}^-=Ver\cap B.\\
\iota F^{-1}0=Ver\cap\partial B.\\
\end{array}
\]
If $Ver\cap B=\emptyset$, then $F(x,y)\geq 0$. A straightforward calculation shows
\[
\bigtriangleup  F(x,y)=16(a_{33}(x^2+y^2)+a_{22}+2b_{23}x-2c_{23}y).
\]
Thus $E=\{(x,y)\in \Bbb{R}^2:\bigtriangleup  F(x,y)=0\}$  is an ellipse.\\

Claim: We have   $ \psi^{-1}C\cap Int(E)=\emptyset$. On the contrary let us assume that there is an $x\in  C\cap Int(E)\neq\emptyset$.  Then there is an open neighbourhood $U$ of  
$x$ contained in  $Int(E)$. Thus $\bigtriangleup  F(x,y) $ is negative on $U$,  {\it i.e.} $F$ is super-harmonic on $U$. However, $F$ attains its minimum 
in $U$, which is a contradiction.\\ 

From the previous claim we conclude  $C$ is contained in the closure  of  $Ext(E)$, therefore   
$\bigtriangleup  F(x,y)\leq 0$  in  $Int(\psi^{-1}C)$. As a consequence, $F$ is subharmonic in $Int(\psi^{-1}C)$. Let  $c$ be the centre of $\psi^{-1}C$ and   $r$ its  radius.   Let    $(r_n)$  be a strictly  increasing  sequence of positive numbers  
 such that $r_n\xymatrix{\ar[r]_{n\rightarrow\infty}&}r$. Let  
 $x_n\in\overline{B_{r_n}(c)}$  be such that  
$$F(x_n)=max\{F(x):x\in\overline{B_{r_n}(c)}\}.$$

Since $F$ is subharmonic in 
$B_{r_n}(c)$ we conclude $x_n\in \partial B_{r_n}(c)$ and  $(F(x_{n}))$ is a strictly increasing sequence of positive numbers. 
Since  $Int(\psi^{-1}C)\cup \psi^{-1}C$ is a compact set, we can assume
 $x_n\xymatrix{\ar[r]_{n\rightarrow\infty}&}x$, and clearly   $x\in \psi^{-1}C$. On the other hand, since $F$ is continuous we conclude 
  $F(x_n)\rightarrow F(x)=0$, which is a contradiction.
\end{proof}

\begin{corollary} \label{l:conpo}
There is a $\gamma_0\in \PSL(3,\Bbb{R})$ such that 
 \begin{enumerate} 
\item \label{l:con1} $\gamma_0\iota \PSL(2,\Bbb{R})\gamma^{-1}_0=\PO^+(2,1)$, where $\PO^+(2,1)$ is the principal connected component of $\PO(2,1)$,
\item \label{l:con2} $\gamma_0 Ver \cap \Bbb{H}^2_\Bbb{C} $ is non-empty and  $\PO(2,1)^+$-invariant. 
\end{enumerate}
\end{corollary}
\begin{proof}
Let us prove  (\ref{l:con1}). By   Lemma \ref{c:liedim} we have  that 
$\iota\PSL(2,\Bbb{R})$ is a Lie group of   dimension   three and   preserves  the quadric in  $\Bbb{P}^2_{\Bbb{R}}$ given by  
\[
\{[w^2,2wz,z^2]: z,w\in \Bbb{R}\}.
\]
 Thus there is a $\gamma_0$ in $\PSL(3,\Bbb{R})$ such that 
  $\gamma_0\iota M\ddot{o}b(\hat{\Bbb{R}})\gamma^{-1}_0$ preserves 
$$\{[x,w,z]\in\Bbb{P}^2_{\Bbb{R}}:\vert y\vert^2+\vert w\vert^2<\vert x\vert^2\}.$$ 
Hence  $\gamma_0\iota\PSL(2,\hat{\Bbb{R}})\gamma^{-1}_0=\PO^+(2,1)$.
 Part  (\ref{l:con2}) is now trivial. 
\end{proof}

\begin{theorem}\label{t:liedim}
Let  $\Gamma\subset\PSL(2,\Bbb{C})$ be a discrete non-elementary  group. The group  $\iota\Gamma$ is complex hyperbolic if and only if $ \Gamma$ is Fuchsian, {\ i.e.} a subgroup of  $\PSL(2,\Bbb{R})$. 
\end{theorem}
\begin{proof}
Assume that  $\iota\Gamma$ preserves a complex ball $B$. Then by Lemma \ref{c:liedim} we deduce that  $\Gamma$ 
preserves a circle $C$ in  the Riemann sphere. Let  $B^+$ and  $B^-$ be the connected  components of $\Bbb{P}_\Bbb{C}^1\setminus C$ and assume that there is a $\tau\in\Gamma$ such that  $\tau(B)^+=B^-$. Let  $x\in  Ver\cap B$  and denote by  $Aut^+(BV)$   the principal connected  component  of  $Aut(BV)$ which contains the identity. Then by Lemma \ref{c:liedim} we  deduce  
\[
\begin{array}{l}
Aut^+(BV)x= \psi\iota^{-1}Aut(BV)\psi^{-1}x=
\psi (B^{+}) \;\;\hbox {and}
\\
Aut^+(BV) \iota\tau(x)= \psi\iota^{-1}Aut(BV)\tau\psi^{-1}x=\psi(B^{-}).\\
\end{array}
\]
Therefore  
$$Ver=Aut^+(BV)x\cup  Aut^+(BV)\iota\tau(x)\cup  C\subset\overline{\Bbb{H}}^2,$$
which is a contradiction. Clearly, this concludes the proof.
\end{proof}

We arrive at the following theorem:

\begin{theorem}
Let $\Gamma\subset\PSL(2,\Bbb{C})$. Then the following claims are equivalent:
\begin{enumerate}
\item The group $\Gamma$ is Fuchsian.
\item  The group $\iota\Gamma$ is complex hyperbolic. 
\item The group $\iota\Gamma$ is  $\Bbb{R}$-Fuchsian
\end{enumerate}

\end{theorem}

\section{Subgroups of  $\PSL(3,\Bbb{R})$ that Leave Invariant a Veronese Curve}
\label{s:riv}

In this section we  characterize those subgroups of $\PSL(3,\Bbb{R})$ which leave invariant a projective copy of $Ver$.

\begin{theorem}
Let $\Gamma\subset \PSL(2,\Bbb{C})$ be a discrete subgroup. Then the following facts are equivalent
\begin{enumerate}
\item The group $\Gamma $ is conjugate to a subgroup of $Mob(\hat{\Bbb{R}})$.
\item The group $\iota\Gamma$ is conjugate to a subgroup of  $\PSL(3,\Bbb{R})$.
\end{enumerate}
\end{theorem}
\begin{proof}
Let $\Gamma$ be a subgroup of  $Mob(\hat{\Bbb{R}})$ and  $\gamma\in \Gamma$. Then
\[
\gamma=\left [\left [
\begin{array}{ll}
i & 0\\
0 & -i\\
\end{array}
\right ]\right ]
\left [\left [
\begin{array}{ll}
a & b\\
c & d\\
\end{array}
\right ]\right ]
\]
where  $a,b,c,d\in \Bbb{R}$ and  $ad-bc=1$. A straightforward calculation shows
that 
\[
\iota\gamma=
\left [\left [
\begin{array}{lll}
-1 & 0 &0\\
0 & 1&\\
0 &0 &-1
\end{array}
\right ]\right ]
\left [ 
\left [
\begin{array}{lll}
a^2 & ab & b^2\\
2ac & ad+bc &2bd\\
c^2 & cd& d^2\\
\end{array}
\right ]\right ],
\]
therefore $\iota\Gamma\subset\PSL(3,\Bbb{R})$.\\

Let us assume that there is a real projective space  $\Bbb{P}$ which is  
$\iota\Gamma$-invariant. Thus, as in Lemma \ref{l:semialg}, we conclude that
$$Aut(PV)=\iota\PSL(2,\Bbb{C})\cap\{g\in \PSL(3,\Bbb{C})\vert g\Bbb{P}=\Bbb{P}\}$$
is a semi-algebraic group. Since  
$\Gamma\subset\iota^{-1}Aut(PV)$, we conclude that $\iota^{-1}Aut(PV)$ is a Lie group  with   positive dimension. From the classification of Lie subgroups of   $\PSL(2,\Bbb{C})$ (see \cite{CS1}), we deduce  that $\iota^{-1}Aut(PV)$ is either conjugate to $Mob(\hat{\Bbb{C}})$ or a subgroup of   $Mob(\hat{\Bbb{R}})$. In order to conclude the proof, observe that
the  group $\iota^{-1}Aut(PV)$ can not be  conjugated to  $Mob(\hat{\Bbb{C}})$. 
In fact,  assume on the  contrary that  $\iota^{-1}Aut(PV)=Mob(\hat{ \Bbb{C}})$. Since    
$Mob(\hat{ \Bbb{C}})$ acts transitively on  $\hat{\Bbb{C}}$ we deduce that   $Aut(PV)$ acts transitively  on  $Ver$. Finally, since 
$\psi(\Lambda(\Gamma))\subset Ver\cap\Bbb{P}$, we deduce   
$Ver\subset\Bbb{P}$, which is a contradiction.
\end{proof}

\section{Examples of Kleinian Groups with Infinite Lines in General Position} \label{s:rep}

Let us introduce the following projection, see \cite{goldman}. For each  
 $z\in\Bbb{C}^3$  let   $\eta$ be the function satisfying  $\eta(z)^2=-<z,z>$  and consider the projection 
$\Pi:\Bbb{H}^2_{\Bbb{C}}\rightarrow\Bbb{H}^2_{\Bbb{R}}$ given by
\[\Pi([z_1,z_2,z_3])=[\overline{\eta(z_1,z_2,z_3)}(z_1,z_2,z_3)+\eta(z_1,z_2,z_3)(\overline{z_1},\overline{z_2},\overline{z_3})].\]

\begin{lemma}
The projection  $\Pi$ is $\PO(2,1)$-equivariant.
\end{lemma}
 \begin{proof}
 Let  $A\in O(2,1)$ and  $[z]\in \Bbb{H}_{\Bbb{C}}^2$. Then  
\[
\begin{array}{ll}
\Pi [Az] &=[\overline{ \eta(Az)}Az+\eta(Az)\overline{Az}]\\
         &=[\overline{\sqrt{-<A z, Az>}}Az+\sqrt {-<Az,Az>}A\bar{z}]\\
         &= [\overline{\sqrt{-<z,z>}}Az+\sqrt{-<z,z>}A\bar{z}]\\
         &= [A][\overline{\eta(z)}z+\eta(z)\overline{z}]\\
         &=[A]\Pi[z].
\end{array}
\]
 \end{proof}
 
 For simplicity in the notation,
in the rest of this article we will write $Ver$ instead of $\gamma_0(Ver)$,  $\psi$ instead of $\gamma_0\circ \psi$, and 
$\gamma_0\iota(\cdot)\gamma_0^{-1}$ instead of $\iota(\cdot)$, where $\gamma_0$ is the element given in Corollary  \ref{l:conpo}.

\begin{lemma} \label{l:prv} 
The map $\Pi:Ver \cap \Bbb{H}^2_{\Bbb{C}}\rightarrow \Bbb{H}^2_{\Bbb{R}}$ is a homeomorphism.
\end{lemma}
\begin{proof} Let us prove that the map is onto.
Let  $x\in\Bbb{H}^+\cup\Bbb{H}^-$ be such that  $\psi(x)\in Ver\cap\Bbb{H}^2_{\Bbb{C}}$. Then 
\[
\begin{array}{ll}
\Bbb{H}^2_{\Bbb{R}}&=\PSO^+(2,1)\Pi(\psi x)\\
&=\Pi(\PSO^+(2,1)\psi x)\\
   &=\Pi(\iota\PSL(2,\Bbb{R}))(\psi(x))\\
&=\Pi(Ver\cap\Bbb{H}^2_{\Bbb{C}})
\end{array}.
\] 
Finally, let us prove that our map is injective. On the contrary, let us assume that there are  $x,y\in Ver\cap\Bbb{H}^2_\Bbb{C}$ such that $\Pi(x)=\Pi(y)$. Now define  
\[
\begin{array}{ll}
H_x=Isot(\PSL(2,\Bbb{R}),\psi^{-1}x),\\
H_y=Isot(\PSL(2,\Bbb{R}),\psi^{-1}y).
\end{array}
\]
Clearly  $H_y$ and $H_x$ are groups where each element is elliptic.  On the other hand, observe that 
\[
\begin{array}{l}
\iota H_x\Pi(x)=\Pi\iota H_x(x)=\Pi(x) \;\; \hbox{and}\\
\iota H_y\Pi(y)=\Pi\iota H_y(y)=\Pi(y).
\end{array}
\]
Therefore  
$$\iota H_x\cup\iota H_y\subset Isot(\PO^+(2,1),\Pi x).$$ 
Since   $\Pi(x)\in \Bbb{H}^2_{\Bbb{R}}$, we deduce that 
 $Isot (\PO^+(2,1),\Pi x)$ is a  Lie group where each element is   elliptic. Therefore 
$H=\iota^{-1}Isot (\PO^+(2,1),\Pi x)\gamma_0$ is a Lie  subgroup 
of $\PSL(2,\Bbb{R})$ where each element is elliptic and    $H_y\cup  H_x\subset H$. From the classification of  Lie subgroups of   $\PSL(2,\Bbb{C})$, we deduce  that $H$ is 
conjugate to a subgroup of  $Rot_\infty$.  Hence  $H_y=H_x$ and so $x=y$.\end{proof}
 
\begin{lemma} \label{t:rf}
Let  $\Gamma\subset \PSL(2,\Bbb{C})$ be a discrete group. Then  $\Gamma$ is conjugate to  a subgroup $\Sigma$ of  $\PSL(2,\Bbb{R})$ such that    
$ \Bbb{H}/\Sigma$ is a compact Riemann surface if and only if    $\iota\Gamma$  is conjugate to a discrete compact surface group of $\PO^+(2,1)$.  
\end{lemma}
\begin{proof} Let $\Gamma\subset \PSL(2,\Bbb{R})$ be a subgroup  acting properly, discontinuously, freely, and with compact quotient on $\Bbb{H}^+$. Let $R$  be a fundamental region for  the action of $\Gamma$ on $\Bbb{H}^+$. We may assume without loss of generality that $\psi(R)\subset\Bbb{H}^2_{\Bbb{C}}$. Thus 
 $\Pi\psi\overline{R} $ is a compact subset of  $\Bbb{H}^2_\Bbb{R}$  satisfying  
$\iota\Gamma\Pi\psi\overline{R}=\Bbb{H}^2_{\Bbb{R}}$
which shows that $\iota\Gamma$ is a discrete compact surface group of  $\PO^+(2,1)$.\\ 

Now let us assume that  $\iota\Gamma$ is  a discrete compact surface group  of  $\PO^+(2,1)$.  Then  $\Gamma\subset\PSL(2,\Bbb{R})$. 
Thus $\iota\Gamma\subset\PO^+(2,1)$ and  
$\Bbb{H}^2_{\Bbb{R}}/\iota\Gamma$ is a  compact surface, see \cite{tengren}. Now, consider the following commutative diagram

\begin{equation} 
\xymatrix{
  \Bbb{H}_\Bbb{R}^2 \ar[r]^{\Pi^{-1}} \ar
[d]^{q_1}    & Ver\cap \Bbb{H}^2_\Bbb{C} \ar[d]^{q_2} \ar[r]^{\psi^{-1}} & \Bbb{H}^+\ar[d]^{q_3}\\
\Bbb{H}^2_{\Bbb{R}}/\iota 
\Gamma
\ar[r]^{\widetilde {\Pi}}   & 
(Ver\cap \Bbb{H}^2_\Bbb{C})/ \iota\Gamma\ar[r]^{\widetilde \psi} &
\Bbb{H}^+/\Gamma
} 
\end{equation}
where $q_1,q_2,q_3$ are the quotient maps,
 $\widetilde {\Pi}(x)=q_2\Pi^{-1}q_1^{-1}x$, and 
 $\widetilde\psi(x)=q_3\psi^{-1}q_2^{-1}(x)$. By  Lemma \ref{l:prv}, we conclude that $\Bbb{H}^2_{\Bbb{R}}/\iota\Gamma,(Ver\cap \Bbb{H}^2_\Bbb{C})/\iota\Gamma,\Bbb{H}^+/\Gamma$ are homeomorphic compact surfaces, which concludes the proof.
\end{proof}

\section*{Proof of theorem \ref{t:main2}}
\begin{proof}
If $\Gamma\subset\PO^+(2,1)$ is a discrete compact surface  group, then  by  Lemma 
\ref{l:conpo} we can assume that there is a $\Sigma\subset\PSL(2,\Bbb{R})$ such that   $\iota\Sigma=\Gamma$. By  Theorem  \ref{t:rf} we know that  $\Bbb{H}/\Sigma$ is a compact Riemann  surface. 
 From  the classic theory of quasi-conformal maps, see  \cite{lipa1, lipa2},  it is known that  there   is a sequence of quasi-conformal maps 
 $(q_n:\widehat{\Bbb{C}}\rightarrow\widehat{\Bbb{C}})$ such that
$q_n\xymatrix{\ar[r]_{n\rightarrow\infty}&}Id$  and  
$\Sigma_n=q_n\Sigma q_n^{-1}$ is a quasi-Fuchsian group, which can not be conjugated to a Fuchsian one. In consequence   $\Gamma_n=\gamma_0\iota\Sigma_n\gamma_0^{-1}$ is  complex Kleinian  and neither conjugate to a subgroup of  $\PU(2,1)$  nor  $\PSL(3,\Bbb{R})$, which concludes the proof.  
\end{proof}

Now the following result is trivial.

\begin{corollary}
There are complex Kleinian groups acting on $\Bbb{P}^2_{\Bbb{C}}$ which are not conjugate to  either a complex hyperbolic group or a virtually affine group.
\end{corollary}

 The authors would like to thank  J. Seade  for 
fruitful conversations.  Also we would like to thank the  staff of UCIM at UNAM for their kindness and help. 

\bibliographystyle{amsplain}

\begin{thebibliography}{10}


\bibitem{BCN}
W. Barrera, A. Cano, and J. P. Navarrete, The limit set of discrete subgroups of PSL(3,C), Math. Proc.
Cambridge Philos. Soc. {\bf 150} (2011), no. 1, pp. 129-146. 


\bibitem{lipa1}
L. Bers, Several Complex Variables I (Maryland 1970),  Lecture Notes in  Mathematics, ch. Spaces of Kleinian groups, pp. 9–34, Springer-Verlag, Berlin, 1970.

\bibitem{lipa2}
L. Bers,  On moduli of Kleinian
groups, Russian Mathematical Surveys {\bf 29} (1974), no. 2, pp. 88-102. 


\bibitem{CNS}
 A. Cano, J. P. Navarrete, and J. Seade,  Complex Kleinian Groups, Progress in Mathematics, no. 303,
Birkhäuser/Springer, Basel, 2013.

\bibitem{CS}

 A. Cano and J. Seade, On the equicontinuity region of discrete subgroups of PU(1,n), J. Geom. Anal. {\bf 20} 
(2010), no. 2, pp. 291-305.




\bibitem{CS1}
 A. Cano and J. Seade, On discrete
groups of automorphism of PSL(3,C), Geometriae Dedicata {\bf 168} (2014), no. 1, pp. 9-60. 


\bibitem{CG}
 S. S. Chen and L. Greenberg, Hyperbolic spaces,  Contributions to Analysis (A Collection of 
Papers Dedicated to  to Lipman Bers), Academic 
Press, New York, 1974, pp. 49-87.



\bibitem{semi}
Myung-Jun Choi and Dong Youp Suh,  Comparison of semialgebraic groups with Lie groups 
and algebraic groups, RIMS Kôkyûroku {\bf 1449} (2005), pp. 12-20. 


\bibitem{goldman}
 W. M. Goldman, Complex Hyperbolic Geometry,  Oxford University Press, New York, 1999.

\bibitem{JM} 
T. Jorgensen,  A. Marden, Algebraic and geometric convergence of Kleinian groups, Math. Scand.  {\bf 66} (1990), pp. 47-72.

\bibitem{kulkarni}
R. S. Kulkarni,  Groups with domains of  discontinuity, Math. Ann. {\bf 237} (1978), no. 3, pp. 253-272. 

\bibitem{lab}
F.  Labourie,   Lectures on Representations
of Surface Groups, Zurich lectures in advanced mathematics,	European Mathematical Society, 2013.


\bibitem{SV3}
 J. Seade and A. Verjovsky, Actions of discrete groups on complex projective spaces, in M. Lyubich, J. W. Milnor, and Y. N. Minsky, (eds.),  Laminations and Foliations in Dynamics, Geometry and Topology, Contemporary 
Mathematics, vol. 269, AMS, Providence, RI, 2001, pp. 155-178.                                       

\bibitem{SV1}
J. Seade and A. Verjovsky, Higher dimensional complex 
Kleinian groups, Math. Ann. {\bf 322} (2002), no. 2,  pp.  279-300.          

\bibitem{SV2}
J. Seade and A. Verjovsky, Complex Schottky Groups, Asterisque, vol. 287,  SMF, Paris, 2003, pp.  251-272. 

\bibitem{tengren}
T. Zhang,  Geometry of the Hitchin component (2015), Ph. D. Thesis, University of Michigan, https://deepblue.lib.umich.edu/handle/2027.42/113605
\end{thebibliography}

\end{document}